\documentclass[12pt]{article}
\usepackage[a4paper]{geometry}
\usepackage{extramath,amscd}
\usepackage[xy]{extratheorem}
\usepackage[cmtip,all]{xy}
\operators{Ext,Pic,Sym,Ind}
\nc\cusp{{\mathrm{cusp}}}
\let\Q\bbQ
\nc\G{\Gamma}
\nc\omm{\boldsymbol{\omega}}
\nc\bull{\bullet}
\nc\Xbar{\overline X}\nc\Ybar{\overline Y}\nc\Ubar{\overline U}
\author{Cosmin Davidescu and Anthony J. Scholl}
\title{Extensions in the cohomology of Hilbert modular
  varieties}
\date{}

\begin{document}
\maketitle
\vspace*{8mm}
\section*{Introduction}

Let $S$ be a Hilbert modular variety (uncompactified) defined over
$\Q$ attached to a totally real field $F$, We assume $S$ is
nonsingular. The $\ell$-adic cohomology of $S$ carries a nontrivial
weight filtration, and one may consider 
the possible extensions of $\Gal(\Qbar/\Q)$-modules thereby arising.

If $F=\Q$ then $S$ is an open modular curve. The only possible
cohomology where a nontrivial extension could arise is degree $1$, and
by the Manin-Drinfeld principle, the $H^1$ is in fact split.

In dimension greater than 1, the cohomology in each degree has at most
two nonzero steps in the weight filtration.  The Manin-Drinfeld
principle still shows that cusp forms cannot give rise to nontrivial
extensions in the cohomology of $S$, but there is the possibility that
nontrivial extensions could arise between the boundary cohomology and
the part of the cohomology coming from $1$-dimensional automorphic
representations.  Caspar \cite{Cas} investigated this in the case of
Hilbert modular surfaces. He computed the extension classes that arise
for the $H^2$, and showed that they are nontrivial, giving an explicit
description via Kummer theory.

In this paper we consider the case of arbitrary $F$. We show (Theorems
\ref{thm:main1} and \ref{thm:main2}) that nontrivial extensions can
occur only in degree $2r-2$, and that in this case the extensions
which arise are nontrivial, and can again be described explicitly
using Kummer theory.

One motivation for this work is the ``plectic conjecture'' of
Nekov\'a\v{r} and the second author \cite{NSintro}. A consequence of
the results proved here is that the Galois action on $H^*(S)$ (for $S$
now a $GL_2(F)$-Shimura variety) extends to the ``plectic Galois
group''; this completes the proof of Proposition 6.6 of
\cite{NSintro}, as explained in the last section. We also indicate how
the same method gives a proof of the analogous statement
\cite[(3.3.11)]{NSPH} in Hodge theory.

After completing this paper we learnt of independent work by J.~
Silliman \cite{Si}, proving results equivalent to Theorem
\ref{thm:main1} and its Hodge-theoretic analogue.

\section{Hilbert modular varieties}

Throught the paper, $F$ will denote a fixed totally real number field
of degree $r>1$, $\frako_F$ its ring of integers, and
$\Sigma=\Hom_{\Q-\mathrm{alg}}(F,\Qbar)$ where $\Qbar$ is the
algebraic closure of $\Q$ in $\C$. (We do not fix a preferred
embedding of $F$ into $\Qbar$). For a field $k$ we write $\G_k$ for
its absolute Galois group (for some algebraic closure, which will be
clear from the context).

For any $k$-scheme $X$ (where $k\subset\Qbar$), we will generally write
$H^*(X,\Ql)=H^*_{\text{\'et}}(X\otimes_k\Qbar,\Ql)$ and
$H^*(X,\Q)=H^*(X(\C),\Q)$, and similar for compact supports, or
sheaves. (We make an exception to this convention in Proposition
\ref{prop:m+u}, where it would cause confusion.)

We let $G\subset R_{F/\Q}GL_2$ be the algebraic subgroup whose group
of $\Q$-points is $G(\Q)=\{ g\in GL_2(F) \mid \det g\in\Q^*\}$.  Until
the last section, $S$ will be a Hilbert-Blumenthal modular variety
over $\Q$ associated to some open subgroup $K\subset
G(\bbA_\Q^\infty)$. We assume that $K$ is sufficently small to ensure
that $S$ is smooth.

\subsection*{The minimal compactification}

This is a compactification
\[
S\mapright{j} S^* \mapleft{i}S^\infty
\]
where $S^*$ is normal and proper, and $S^\infty$ is
zero-dimensional. We have the long exact sequence
of cohomology:
\begin{equation}\label{eq:bdyseq}
  H^n_c(S,\Ql) \to H^n(S,\Ql)
  \to H^n(S^\infty,i^*Rj_*\Ql) \to \dots
\end{equation}
Write $H^n_!(S,\Ql)=\im(H^n_c(S,\Ql) \to H^n(S,\Ql))$ for the interior
cohomology, $H^n_\partial(S,\Ql)=H^n(S^\infty,i^*Rj_*\Ql)$ for the
boundary cohomology. The exact sequence is auto-dual, via Poincar\'e
duality betwen $H(S)$ and $H_c(S)$, and the duality on boundary
cohomology $H^n_\partial(S,\Ql)^\vee\simeq
H^{2r-1-n}_\partial(S,\Ql)(r)$.

The boundary cohomology is independent of the choice of
compactification; both it and the exact sequence \eqref{eq:bdyseq} can
be computed using singular cohomology of the Borel-Serre
compactification, as was first done by Harder \cite{HaSL2}, who showed
that for $n=1$, $2r-1$ one has
$H^n_!(S,\Q)=0$, and the sequence splits into short exact sequences

 --- for $2\le n<r$:
\begin{equation}\label{eq:excomp}
  0 \to H^{n-1}_\partial(S,\Q) \to 
  H^n_c(S,\Q) \to
  H^n_!(S,\Q)=H^n(S,\Q) \to 0
\end{equation}
--- for $r<n\le 2r-2$:
\[
0 \to H^n_c(S,\Q) =H^n_!(S,\Q)
\to H^n(S,\Q)
\to  H^n_\partial(S,\Q) \to 0
\]
 --- in middle degree: 
\[
\xymatrix
@R=10pt@C=10pt
{0\ar[r] & H^{r-1}_\partial(S,\Q)\ar[r] &
H^r_c(S,\Q) \ar[rr]\ar@{->>}[rd]& & H^r(S,\Q)\ar[r]&
  H^r_\partial(S,\Q) \ar[r] & 0
\\
&&& H^r_!(S,\Ql)\ar@{^{(}->}[ru]
}
\]
By the comparison isomorphism, the same holds for $\ell$-adic cohomology.

It is also shown in \cite{HaSL2} that for any $z\in S^\infty(\C)$ the
boundary cohomology at $z$ satisfies
\begin{align*}
H^1((Rj_*\Q)_z) & = \Hom(\frako_F^*,\Q)\\
H^n((Rj_*\Q)_z) & = \bigwedge^n H^1((Rj_*\Q)_z)\quad\text{for $1\le
                  n\le r-1$.}
\end{align*}

\subsection*{The toroidal compactification}

This is a smooth, projective compactification $S\inject{} \tilde S$
whose boundary $\tilde S^\infty$ is a divisor with strict normal
crossings. It depends on a choice of admissible cone decomposition
of the cone of totally positive elements $(F\otimes_\Q\R)_+
\subset F\otimes_\Q\R$
(see \cite[\S4]{Ra} or \cite[\S4.1.4]{Hi}). The boundary component $\tilde S^\infty_y$ over $y\in
S^\infty$ is the quotient $Z_y/\Delta_y$, where $Z_y$ is a reduced
scheme locally of finite type over $k(y)$, whose irreducible
components $Z_{y,\sigma}$ are smooth toric varieties, and $\Delta_y\subset
\frako_F^*$ is a torsion-free subgroup of finite index.

The varieties $Z_{y,\sigma}$ have vanishing $H^1$. So by Meyer--Vietoris
$H^1(Z_y)$ equals the $H^1$ of the nerve of the cover $\{Z_{y,\sigma}\}$.
This nerve, being the simplicial complex associated to the cone decomposition
of $(F\otimes_\Q\R)_+$, is contractible, so $Z_y$ has vanishing $H^1$.
It follows that $H^1(\tilde
S^\infty,\Ql)=H^0(S^\infty,\Ql)\otimes_\Q \Hom(\frako_F^*,\Q)$ as 
$\Gamma_\Q$-modules, and
that the natural homomorphism $H^1(\tilde S^\infty,\Ql) \to
H_\partial^1(S,\Ql)$ is an isomorphism, giving isomorphisms
\[
H^n_\partial(S,\Ql) \simeq H^0(S^\infty,\Ql)\otimes_\Q
\bigwedge^n\Hom(\frako_F^*,\Q)
\]
for $1\le n\le r-1$.

One also has 
$H^2_!(S,\Ql)=\im(H^2(\tilde S,\Ql) \to H^2(S,\Ql)) = W_2 H^2(S,\Ql)$ and so the exact
sequence \eqref{eq:excomp} for $n=2$ may be rewritten as
\begin{equation} \label{eq:altext}
  0 \to H^1(\tilde S^\infty,\Ql) \to H^2_c(S,\Ql) \to \im(H^2(\tilde
  S,\Ql) \to H^2(S,\Ql)) \to 0.
\end{equation}

Define, for any $y\in S^\infty(\Qbar)$, $\Pic^0 \tilde S^\infty_y =
\ker(\Pic \tilde S^\infty_y \to H^2(\tilde S^\infty_y,\Ql(1))$. One
then has:
\begin{lem}\label{lem:picbdy}
  \[
  \Pic^0 \tilde S^\infty_y = \ker(\Pic \tilde S^\infty_y\to \Pic Z_y)
  \simeq \Hom(\Delta_y, \Qbar^*).
  \]
\end{lem}

\begin{proof}
  By the above discussion of the toroidal boundary we have
  $H^0(Z_y,\Ql)=\Ql$ and  $H^1(Z_y,\Ql)=0=\Pic^0 Z_y$.
  So from the Cartan spectral
  sequences for $Z_y \to \tilde S^\infty_y$ with coefficients in $\mathbb{G}_\mathrm{m}$
  we obtain the exact rows of the commutative diagram
  \[
  \xymatrix
  @R=20pt@C=20pt
  {  0\ar[r]& H^1(\Delta_y,\Qbar^*)\ar[r]\ar@{.>}[d]& \Pic \tilde
    S^\infty_y\ar[r]\ar[d] &
    (\Pic Z_y)^{\Delta_y}\ar@{_{(}->}[d] \\
    0 \ar[r] & H^2(\Delta_y,\Ql)(1) \ar[r]& H^2(\tilde
    S^\infty_y,\Ql(1)) \ar[r] &
    H^2(Z_y,\Ql(1))^{\Delta_y}     
  }
  \]
  in which the right hand vertical arrow is injective. If $x\in
  H^1(\Delta_y,\Qbar^*)$ then its image in $H^2(\Delta_y,\Ql)(1)$ is
  fixed by an open subgroup of $\Gamma_\Q$, hence the left hand
  vertical map is zero. The result then follows by the snake lemma.
\end{proof}

\subsection*{The line bundles $\calL_\tau$}

For each $\tau\colon F \to L\subset \C$ there is an invertible sheaf
$\calL_\tau$ on $S\otimes L$, with the property that the sections of
$\bigotimes\calL_\tau^{k_\tau}$ are the modular forms of weight
$(k_\tau)$ (see \cite[6.9(b)]{Ra}). If $L=\C$ this is the usual line
bundle on $S(\C)$ associated to the factor of automorphy $(\gamma,z)
\mapsto (c_\tau z_\tau + d_\tau)$ \cite[6.15]{Ra}, and extends to the
toroidal compactification in a unique way such that the pullback to
each $Z_y$ is trivial (see \cite[\S II.7]{vdG}, which is for the case
$r=2$, but the general case is the same). By its very definition and
the previous lemma, the restriction of $\calL_\tau$ to $\tilde
S^\infty_y$, $y\in S^\infty(\Qbar)$, lies in $\Pic^0\tilde S^\infty_y$
and can be identified (up to a sign independent of $\tau$) with the
homomorphism $\tau \in H^1(\Delta_y,\Qbar^*)\subset
\Hom(F^*,\Qbar^*)$.

By definition, the Galois action on the line bundles is given by
$\sigma_*\calL_\tau =\calL_{\sigma\tau}$ for $\sigma\in\Gamma_\Q$.

Write $\eta_\tau\in H^2_!(S,\Ql(1))$ for the cohomology class of
$\calL_\tau$. The classes $\eta_\tau$ are linearly independent and
therefore generate a subspace  isomorphic to the permutation
representation $\Ql[\Sigma]$.

\section{The extension classes}

For $I\subset \Sigma$ with $0<\#I=m<r$, let $\eta_I=\bigwedge_{\tau\in
  I}\eta_\tau\in H^2_!(S,\Ql(1))$, and let
\[
H^{2m}_A(S,\Ql)= \sum_{\#I=m} H^0(S,\Ql)\cup \Ql(-m)\eta_I\subset
H^{2m}_!(S,\Ql)
\]
 From \cite{HaSL2} one has the following description of the interior cohomology.
\begin{itemize}
\item For $n\ne r$ odd or $n=2r$, $H^n_!(S,\Ql)=0$.
\item For $0<n=2m<2r$, $n\ne r$, $H^{2m}_!(S,\Ql)=H^{2m}_A(S,\Ql)$.
\item If $r=2m$ is even then $H^r_!(S,\Ql)=H^r_A(S,\Ql)\oplus
  H^r_{\cusp}(S,\Ql)$, a direct sum of $\Gal(\Qbar/\Q)$-modules stable
  under the Hecke algebra.
\end{itemize}

\begin{thm} \label{thm:main1}
  (i) For $2<n\le r$, there is a unique splitting of $\Gal(\Qbar/\Q)$-modules:
  $H^n_c(S,\Ql)= H^n_!(S,\Ql)\oplus H^{n-1}_\partial(S,\Ql)$. 

  (ii) For $r \le n< 2r-2$, there is a unique splitting of $\Gal(\Qbar/\Q)$-modules:
  $H^n(S,\Ql)= H^n_!(S,\Ql)\oplus H^n_\partial(S,\Ql)$. 
\end{thm}

\begin{proof}
  As (i) and (ii) are equivalent by Poincar\'e duality, it is enough
  to prove (i). Because  $H^{n-1}_\partial(S,\Ql)$ is pure of weight
  $0$, there is at most one splitting.

  If $n<r$ is odd, there is nothing to prove as $H^n_!=0$ for $n<r$. 
  
  If $n=r$ then by Manin-Drinfeld principle, the extension splits over
  $H^r_\cusp$. So it is enough in every case to split the extension
  over $H^n_A\subset H^n_!$. Therefore (i) will follow from:

  \begin{prop}
    Let $1<m\le r/2$. Let $\cup^m H^2_c(S,\Ql)\subset H^{2m}_c(S,\Ql)$
    be the image of $\otimes^m H^2_c(S,\Ql)$ under the
    cup product. Then the composite
    \[
    \cup^m H^2_c(S,\Ql) \inject{} H^{2m}_c(S,\Ql) \surject{}
    H^{2m}_!(S,\Ql)
    \]
    is an isomorphism if $m<r/2$; for $r=2m$ even, it gives an
    isomorphism
    \[
    \cup^{r/2} H^2_c(S,\Ql)\isomarrow H^r_A(S,\Ql) \subset
    H^r_!(S,\Ql).
    \]
\end{prop}

\begin{proof}
  It is enough to check that the composite
  \[
  \textstyle\bigotimes^m H^2_c(S,\Ql)  \mapright{\cup} H^{2m}_c(S,\Ql)
  \to H^{2m}(S,\Ql) 
  \]
  has image $H^{2m}_A$, and is zero on elements $\otimes x_i$ where some
  $x_i$ is in the image of the boundary cohomology. The first
  assertion is clear as the cup product map $\cup\colon\bigotimes^m H^2_A(S,\Ql) \surject{}
  H^{2m}_A(S,\Ql)$ is surjective.  As for the second, we have a commutative diagram:
  \[
  \xymatrix
  @R=20pt@C=30pt
   { H^1_\partial(S,\Ql)\otimes H^2_c(S,\Ql) \ar[r]^-{\partial^1\otimes
      id} \ar[rd]^0 &
    H^2_c(S,\Ql)\otimes H^2_c(S,\Ql) \ar[r]^-\cup \ar[d] & H^4_c(S,\Ql) \\
    & H^2(S,\Ql)\otimes H^2_c(S,\Ql)\ar[ru]^-\cup
  }
  \]
  and therefore the composite of the horizontal arrows
  $H^1_\partial\otimes H^2_c \to H^4_c$ is zero. Therefore, if
  $\eta_\tau^c\in H^2_c(\Ql(1))$ are any classes lifting $\eta_\tau\in
  H^2_A(S,\Ql(1))$, then
  \[
  \cup^m H^2_c(S,\Ql(1)) = \cup^m \langle\, \{\eta^c_\tau\}\, \rangle
  =\bigoplus_{\#I=m} (\textstyle\bigwedge_{\tau\in I} \eta_\tau^c)\Ql
  \isomarrow H^{2m}_A(S,\Ql(m)).  \qedhere
  \]
\end{proof}
\def\qedsymbol{}
\vspace*{-6ex}
\end{proof}
For the second result we need some more notation. Consider the Kummer
homomorphism $\kappa_F\colon F^* \to H^1(\G_F,\Ql(1))$. Composing with
the isomorphism given by Shapiro's lemma:
\[
H^1(\G_F,\Ql(1)) \isomarrow H^1(\G_\Q,\Ql^\Sigma(1))
\]
(which does not depend on a choice of embedding $F\subset \Qbar$) we
obtain a homomorphism
\[
\kappa'_F \colon F^* \to H^1(\G_\Q,\Ql^\Sigma(1))
\]
inducing an isomorphism between the completed tensor product
$F^*\hat\otimes\Ql$ and the right hand side.

The morphism of $0$-dimensional schemes $\varepsilon \colon S^\infty \to
\pi_0(S)$ gives a $\G_\Q$-equivariant map
\[
\varepsilon^*\colon H^0(S,\Ql) \to H^0(S^\infty,\Ql).
\]

\begin{thm}\label{thm:main2}
  Assume $r>2$. Consider the extension
  \[
  \begin{CD}
    0 @>>> H^1_\partial(S,\Ql) @>>> H^2_c(S,\Ql) @>>> H^2_!(S,\Ql)
    @>>> 0 \\[-1ex]
     && \Vert && && \Vert \\[-0.5ex]
    @. \hspace*{-15mm}\Hom(\frako_F^*, H^0(S^\infty,\Ql))\hspace*{-15mm} @. @.
    \hspace*{-15mm}H^0(S,\Ql)\otimes \Ql[\Sigma](-1)\hspace*{-15mm}
  \end{CD}
  \]
  Its class is the image of $\varepsilon^*\otimes \kappa'_F$ under the
  map
  \begin{gather*}
    \Hom_{\G_\Q}(H^0(S,\Ql),H^0(S^\infty,\Ql)) \otimes
    \Hom(\frako_F^*,H^1(\G_\Q,\Ql^\Sigma(1))) \\
    \downarrow \\
    \Hom(\frako_F^*,
    H^1(\G_\Q,\Hom(H^0(S,\Ql),H^0(S^\infty,\Ql)\otimes\Ql^\Sigma(1)))
    \\ 
    \Vert \\
    \Ext^1_{\G_\Q}(H^2_!(S,\Ql), H^1_\partial(S,\Ql))\,.
  \end{gather*}
\end{thm}

\begin{rems}
(i) By duality, the same class classifies the extension in cohomology
without support
\[
0 \to H^{2r-2}_!(S,\Ql) \to H^{2r-2}(S,\Ql) \to
H^{2r-2}_\partial(S,\Ql) \to 0.
\]
(ii) The analogous result for $r=2$ is proved in \cite{Cas} by a different
method; the same proof as given below also works in this case with
minor modification.
\end{rems}

\begin{proof} The extension class is determined by its restriction to
  any open subgroup of $\G_\Q$. Let $k\subset \Qbar$ be a number field
  containing a Galois closure of $F$, for which $\G_k$ acts trivially
  on $\pi_0(S\otimes\Qbar)$ and $S^\infty(\Qbar)$. For each
  connected component $S'\subset S\otimes_\Q k$ and for each
  $\tau\in\Sigma=\Hom_{\Q-\mathrm{alg}}(F,k)$,  consider the pullback $E(S',\eta)$
  in the diagram of $\G_k$-modules
  \[
  \begin{CD}
    0 @>>> H^1_\partial(S',\Ql) @>>> H^2_c(S',\Ql) @>>> H^2_!(S',\Ql)
    @>>> 0 \\[-1ex]
    && \Vert && \uparrow && \cup \\[-0.5ex]
    0 @>>> \Hom(\frako_F^*, H^0(S^{\prime\infty},\Ql)) @>>> E(S',\eta)
    @>>> \Ql(-1)\eta_\tau @>>> 0
  \end{CD}
  \]
  It is then enough to check that the extension class of each
  $E(S',\eta)$ in 
  \[
  \Ext^1_{\G_k}(\Ql(-1), \Hom(\frako_F^*, H^0(S^{\prime\infty},\Ql)))
  = \Hom\bigl(\frako^*_F, H^1(\G_k, \Ql(1))^{S^{\prime\infty}}\bigr)
  \]
  is (up to sign independent of $S'$ and $\eta$) given by the
  composite homomorphism
  \[
  \frako^*_F \xrightarrow{\tau_*\circ\kappa_F} H^1(\G_k,\Ql(1))
  \inject{\mathrm{diag}} H^1(\G_k, \Ql(1))^{S^{\prime\infty}}
  \]
  For this, we use the alternative description \eqref{eq:altext} of
  the extension, which then puts us in the following general
  situation. Let $k$ be any field of characteristic different from
  $\ell$, $X/k$ smooth and proper, $i\colon Y \inject{} X$ the
  inclusion of a reduced divisor, and $U=X\setminus Y$. To avoid ambiguity
  we temporarily change notation in order to distinguish between the
  $\ell$-adic cohomology of $\Xbar=X\otimes_k\bar k$ and that of $X$,
  and likewise for $Y$.

  Let $\calL\in\Pic X$ such that $0=cl_{\Ybar}(i^*\calL)\in
  H^2(\Ybar,\Ql(1))$.   We then obtain by pullback an
  extension $E_\calL$ of $\G_k$-modules:
  \[
  \xymatrix
  @R=20pt@C=15pt
  {
    0\ar[r]& \coker\bigl( H^1(\Xbar,\Ql) \to
    H^1(\Ybar,\Ql)\bigr)\ar[r] \ar@{=}[d]&
    H^2_c(\Ubar,\Ql) \ar[r]& H^2(\Xbar,\Ql) \ar[r]^-{i^*} & H^2(\Ybar,\Ql)
    \\
    0 \ar[r]& \coker\bigl( H^1(\Xbar,\Ql) \to H^1(\Ybar,\Ql)\bigr)\ar[r]
    &
    E_\calL \ar[r] \ar[r] \ar@{.>}[u]  &
    \Ql(-1)\ar[r]\ar[u]^{cl_{\Xbar}(\calL)}
    \ar[ur]_-{cl_{\Ybar}(\calL)=0}  &  0
  }
  \]
  and thus an extension class $e_\calL\in H^1(\G_k,\coker (
  H^1(\Xbar,\Ql) \to H^1(\Ybar,\Ql))(1))$. 
  
  \begin{prop} \label{prop:m+u}
    $e_\calL$ equals the image
    of $i^*\calL$ under the composite map
    \[
    \Pic^0 Y \xrightarrow{AJ_Y} H^1(\G_k,H^1(\Ybar,\Ql)(1)) \to
    H^1(\G_k,\coker (
    H^1(\Xbar,\Ql) \to H^1(\Ybar,\Ql))(1))
    \]
    where $AJ_Y$ is the $\ell$-adic Abel-Jacobi map.
  \end{prop}
  
Recall that $AJ_Y$ is defined to be the composite of the following two
maps:
\begin{itemize}
\item the Chern class 
  \[
  \Pic^0 Y \to \Fil^1H^2(Y,\Ql(1)) =
  \ker\bigl(H^2(Y,\Ql(1)) \to   H^2(\Ybar,\Ql(1))^{\G_k} \bigr)
  \]
\item the map 
  \[
  \Fil^1H^2(Y,\Ql(1))  \to H^1(\G_k,H^1(\Ybar,\Ql)(1))
  \]
  coming from the Hochschild-Serre spectral sequence in continuous
  $\ell$-adic cohomology.
\end{itemize}

  \begin{proof}
    Apply \cite[9.5]{Ja} to the triangle
    \[
    R\Gamma(\Xbar,\Ql) \mapright{i^*} R\Gamma(\Ybar,\Ql) \to
    R\Gamma_c(\Ubar,\Ql)[1] \to R\Gamma(\Xbar,\Ql)[1]
    \]
    to get the commutative pentagon:
    \[
    \xymatrix
    @R=25pt@C=1pt
    {
      H^0(\G_k,\ker H^2(i^*)(1))\ar[rr]&& H^1(\G_k,\coker H^1(i^*)(1))
      \\
      \ker\bigl(H^2(X,\Ql(1)) \to H^2(\Ybar,\Ql(1))^{\G_k}\bigr)\ar[u]
      \ar[rr] \ar[dr]^-{i^*} &&  H^1(\G_k,H^1(\Ybar,\Ql)(1))\ar[u]
      \\
      &\hspace*{-2cm}\ker\bigl( H^2(Y,\Ql(1)) \to
      H^2(\Ybar,\Ql(1))^{\G_k}\bigr)\hspace*{-2cm} 
      \ar[ru]
    }
    \]
    under which the various cohomology classes of $\calL$ are mapped
    as follows:
    \[
    \xymatrix
    @R=20pt@C=25pt
    {  cl_{\Xbar}(\calL)\ar@{|->}[rr] && e_{\calL}
      \\
      cl_X(\calL) \ar@{|->}[rr] \ar@{|->}[u]&& AJ_Y(i^*\calL)
    }
    \]
    The commutativity then gives the desired result.
  \end{proof}

To apply this in our situation, take $k$ as above, $X=\tilde
S'\otimes_\Q k$, $Y=\tilde S^{\prime\infty}\otimes_\Q k$, and
$\calL=\calL_\tau$. We have seen that for each $y\in S^\infty(k)$ the
restriction $\calL_\tau|_{\tilde S^{\infty}_y} \in\Pic^0(\tilde
S^\infty_y)=\Hom(\Delta_y,k^*)$ (using the isomorphism of
Lemma \ref{lem:picbdy}) is the map $\tau$ (up to a sign independent
of $y$ and $\tau$). The result then follows from the commutative
diagram
\[
\begin{CD}
\Pic^0(\tilde S^\infty_y) @>{AJ}>> \Fil^1 H^2(\tilde
S^\infty_y,\Ql(1)) & \subset & H^2(\tilde S^\infty_y,\Ql(1))
 \\
@AA{\wr}A @AA{\wr}A
 \\
H^1(\Delta_y, H^0(Z_y,\mathbb{G}_\mathrm{m}) @. 
H^1(\Delta_y, H^1(Z_y,\Ql(1))
 \\
\Vert && \Vert
 \\
H^1(\Delta_y,k^*) @>{H^1(\kappa_k)}>> H^1(\Delta_y,H^1(\G_k,\Ql(1))
\end{CD}
\]
where the right hand vertical isomorphism comes from the Cartan
spectral sequence for $Z_y \to \tilde S^\infty_y$.
\end{proof}

\section{Further remarks}

We may perform the same computations in Hodge theory. The proof of
the splitting in Theorem \ref{thm:main1} carries through without change. For the
proof of Theorem \ref{thm:main2}, one should replace absolute
$\ell$-adic cohomology with absolute Hodge cohomology
\cite{BeAH}. Then the Kummer homomorphism $\kappa'_F$ is replaced by
the archimedean regulator map
\begin{align*}
\kappa'_{F,\mathcal{H}} \colon F^* 
& \to H^1_{\mathcal H}(\Spec\C,\R(1)^\Sigma) = \R^\Sigma\\
x& \mapsto (\log\abs{\tau(x)})_{\tau\in\Sigma}.
\end{align*}
This gives a proof of Theorem (3.3.11) of \cite{NSPH}.  An alternative
approach is to use explicit formulae for Eisenstein cohomology, as
done in the case $r=2$ in \cite{Cas}; details will
appear elsewhere.

Suppose now that $S=S_K$ is a Shimura variety for the full group
$GL_2/F$, where $K\subset GL_2(\bbA_F^\infty)$ is a sufficiently small
open compact subgroup.  Theorems \ref{thm:main1} and \ref{thm:main2}
are equally valid in this setting.  We can now complete the proof of
the relevant part of Proposition 6.6 of \cite{NSintro}:

\begin{cor}
  There exists an action of $\Gamma_F^{pl}$ on $H^*(S,\Qlbar)$,
  extending the action of $\G_\Q$.
\end{cor}

\begin{proof}
We recall some definitions and facts from \cite{NSintro} concerning the
``plectic Galois group'', which is the group
\[
\G_F^{pl}=\Aut(F\otimes_\Q\Qbar/F).
\]
It canonically contains $\G_\Q$ as a subgroup. After fixing
embeddings $\bar \tau \colon \overline F \to \Qbar$ extending
$\tau\in\Sigma$ one obtains an isomorphism with the wreath product
\[
\G_F^{pl} \isomarrow \G_F^\Sigma \ltimes \Sym(\Sigma).
\]
The homomorphism $\G_F^\Sigma \ltimes \Sym(\Sigma)\to
\G_F^{\mathrm{ab}}$ which is trivial on the symmetric group and on
each copy of $\G_F$ is the obvious quotient defines a homomorphism
$\G_F^{pl}\to G_F^{\mathrm{ab}}$ which does not depend on choices, and
whose restriction to $\G_\Q$ is the transfer homomorphism
$\mathrm{Ver}\colon \G_\Q \to \G_F^{\mathrm{ab}}$. 

The action on $\G_\Q$ on both $\pi_0(S\otimes\Qbar)$ and
$S^\infty(\Qbar)$ factors\footnote{
  It is here that we use the fact that $S$ is a $GL_2$-Shimura
  variety. For the varieties considered earlier, this is false; see
  \cite[(0.3)]{NCM}.}
 through $\mathrm{Ver}$, and so extends to
$\G_F^{pl}$. The subspace of $H^2_!(S,\Ql)$ spanned by the classes
$\eta_\tau$ is the induced representation
$\Ind_{\G_F}^{\G_\Q}\Ql(-1)=\Ql(-1)^\Sigma$, and more generally the
subspace of $H^{2m}_!(S,\Ql)$ spanned by the products $\eta_I$ is the
degree $m$ part of the tensor induction
$(\Ql(0)\oplus\Ql(-1))^{\otimes\Sigma}$, with $\Ql(-i)$ in degree
$i$, so extends (canonically) to a representation of $\G_F^{pl}$. It
follows from all of this that there is a canonical action of the
plectic Galois group on $H^*_\partial(S,\Ql)$ and $H^*_A(S,\Ql)$.

The main result of \cite{Nss} shows that $H^r_{\cusp}(S,\Qlbar)$ is a
sum of tensor inductions of $2$-dimensional representations of
$\Gamma_F$, and therefore carries a (noncanonical) action of the
plectic Galois group extending that of $\Gamma_\Q$. To complete the
proof, in view of Theorems \ref{thm:main1} and \ref{thm:main2} it is
therefore enough to show that the action of $\G_\Q$ on $H^2_!(S,\Ql)$
can be extended to the plectic Galois group. Since $\varepsilon^*$ is
$\G_F^{pl}$-equivariant, this follows from:

\begin{lem}
  The restriction homomorphism 
  \[
  H^1(\G_F^{pl},\Ql(1)^\Sigma)  \to
  H^1(\G_\Q,\Ql(1)^\Sigma)= H^1(\G_F,\Ql(1))
  \]
  is an isomorphism.
\end{lem}

This is a consequence of the K\"unneth formula:
\begin{align*}
H^1(\G_F^{pl},\Ql(1)^\Sigma) \simeq
H^1(\G_F^\Sigma,\Ql(1)^\Sigma)^{\Sym(\Sigma)}
&= \bigl(H^1(\G_F,\Ql(1))^{\Sigma}\bigr)^{\Sym(\Sigma)}
\\
&=H^1(\G_F,\Ql(1)). \qedhere
\end{align*}
\end{proof}

\vspace*{10mm}
 \noindent\textsc{Department of Pure Mathematics and Mathematical Statistics\\
   Centre for Mathematical Sciences\\
   Wilberforce Road\\
   Cambridge\enspace CB3 0WB
}

\end{document}